\documentclass[11pt]{article}

\usepackage{amsmath,amssymb}
\usepackage{latexsym}
\usepackage{graphicx}
\usepackage{bm}

\newtheorem{thm}{Theorem}

\newenvironment{proof}{\begin{trivlist}
                       \item[]{\bf Proof.}
                       \hspace{0cm}}{\hfill $\Box$
                       \end{trivlist}}

\begin{document}
\title{Nonlinear differential inequality}

\author{N. S. Hoang$\dag$\footnotemark[1] \quad 
A. G. Ramm$\dag$\footnotemark[3]
\\
\\
$\dag$Mathematics Department, Kansas State University,\\
Manhattan, KS 66506-2602, USA
}

\renewcommand{\thefootnote}{\fnsymbol{footnote}}
\footnotetext[1]{Email: nguyenhs@math.ksu.edu}
\footnotetext[3]{Corresponding author. Email: ramm@math.ksu.edu}
\date{}
\maketitle

\begin{abstract} \noindent 
A nonlinear inequality is formulated in the paper.
An estimate of the rate of growth/decay of solutions to this inequality is 
obtained.
This inequality is of interest in a study of dynamical systems 
and nonlinear evolution equations. It can be applied to
a study of global existence of solutions to nonlinear PDE.


\end{abstract}

\section{Introduction}
In this paper the following nonlinear differential inequality 
\begin{equation}
\label{neq1} 
\dot{g}(t) \le -\gamma(t)g(t) + \alpha(t,g(t)) +
\beta(t),\qquad t\ge t_0, \quad \dot{g}=\frac{dg}{dt},\quad g\geq 0,
\end{equation} 
is studied. 
In equation \eqref{neq1}, $\beta(t)$ and 
$\gamma(t)$ are continuous
functions, defined on $[t_0,\infty)$, where $t_0\ge 0$ is a 
fixed
number. 

Inequality \eqref{neq1} was studied in \cite{R499} with 
$\alpha(t,y) = \tilde{\alpha}(t)y^2$, 
where $0\le \tilde{\alpha}(t)$ is a continuous function on $[t_0,\infty)$. 
This inequality arises in the study of the Dynamical Systems Method 
(DSM) for solving nonlinear operator equations.
Sufficient conditions on $\beta$, $\alpha$ and $\gamma$ which yields an 
estimate 
for the rate of growth/decay of $g(t)$ were given in \cite{R499}. 
A discrete analog of \eqref{neq1} was studied in \cite{R538}. 
An application to the study of a discrete version of the DSM for solving 
nonlinear equation was demonstrated in \cite{R538}. 

In \cite{R558} inequality \eqref{neq1} is studied in the case
$\alpha(t,y)=\tilde{\alpha}(t)y^p$, where $p>1$ and $0\le
\tilde{\alpha}(t)$ is a continuous function on $[t_0,\infty)$.  This
equality allows one to study the DSM under weaker
smoothness assumption on $F$ than in the cited works. It allows 
one to study
the convergence of the DSM under the assumption that $F'$ is locally
H\"{o}lder continuous.  An application to the study of large time behavior
of solutions to some partial differential equations was outlined in
\cite{R558}.

In this paper we assume that $0\le\alpha(t,y)$ is a nondecreasing function
of $y$ on $[0,\infty]$ and is continuous with respect to $t$ on 
$[t_0,\infty)$.  Under this
weak assumption on $\alpha$ and some assumptions on $\beta$ and 
$\gamma$,
we give an estimate for the rate of growth/decay of $g(t)$ as $t\to \infty$ in
Theorem~\ref{thm1}.  

A discrete version of \eqref{neq1} is
studied and the result is 
stated in Theorem~\ref{thm2}. In
Section~\ref{sec3} an application of inequality \eqref{neq1} to the study
of large time behavior of solutions to some partial equations is sketched.

\section{Main results}
\label{sec2}
Throughout the paper let us assume that the function $0\le \alpha(t,y)$ is  
locally Lipchitz-continuous, nondecreasing with respect to $y$, and is  
continuous with respect to $t$ on $[t_0,\infty)$.

\begin{thm}
\label{thm1}
Let $\beta(t)$ and $\gamma(t)$ be continuous 
functions on $[t_0,\infty)$.
Suppose there exists a function $\mu(t)>0$, $\mu\in C^1[t_0,\infty)$,
such that
\begin{align}
\label{1eq4}
\alpha\bigg{(}t,\frac{1}{\mu(t)}\bigg{)}+ \beta(t) &\le 
\frac{1}{\mu(t)}\bigg{[}\gamma -
\frac{\dot{\mu}(t)}{\mu(t)}\bigg{]}, \qquad t\geq t_0.
\end{align}
Let $g(t)\ge 0$ be a solution to inequality \eqref{neq1} such that
\begin{equation}
\label{1eq6}
\mu(t_0)g(t_0)    < 1.
\end{equation}
Then $g(t)$ exists globally and the following estimate holds:
\begin{equation}
\label{3eq10}
0\le g(t) < \frac{1}{\mu(t)},\qquad \forall t\ge t_0.
\end{equation}
Consequently, if $\lim_{t\to\infty} \mu(t)=\infty$, then
\begin{equation}
\lim_{t\to\infty} g(t)= 0.
\end{equation}
\end{thm}

\begin{proof}
Denote 
\begin{equation}
\label{eqqe1}
v(t):=g(t)e^{\int_{t_0}^t\gamma(s)ds}.
\end{equation} 
Then inequality 
\eqref{neq1} takes the form
\begin{equation}
\label{eq6}
\dot{v}(t) \le a(t)\alpha\big{(}t,v(t)e^{-\int_{t_0}^t\gamma(s)ds}\big{)} 
+ b(t),\qquad v(t_0)=g(t_0):= g_0,
\end{equation}
where
\begin{equation}
a(t) = e^{\int_{t_0}^t\gamma(s)ds},\qquad b(t):=\beta(t) e^{\int_{t_0}^t\gamma(s)ds}.
\end{equation}
Denote 
\begin{equation}
\label{eq8}
\eta(t) = \frac{e^{\int_{t_0}^t\gamma(s)ds}}{\mu(t)}.
\end{equation}
From inequality \eqref{1eq6} and relation \eqref{eq8} one gets
\begin{equation}
\label{eq9}
v(t_0)=g(t_0) < \frac{1}{\mu(t_0)}=\eta(t_0).
\end{equation}
It follows from the inequalities \eqref{1eq4}, \eqref{eq6} 
and \eqref{eq9} that
\begin{equation}
\label{eq10}
\begin{split}
\dot{v}(t_0) &\le \alpha(t_0,\frac{1}{\mu(t_0)})
+ \beta(t_0)
\le \frac{1}{\mu(t_0)}\bigg{[}\gamma -
\frac{\dot{\mu}(t_0)}{\mu(t_0)}\bigg{]}
 = \frac{d}{dt} \frac{e^{\int_{t_0}^t\gamma(s)ds}}
{\mu(t)}\bigg{|}_{t=t_0} = \dot{\eta}(t_0).
\end{split}
\end{equation}
From the inequalities \eqref{eq9} and \eqref{eq10} it follows 
that there exists $\delta>0$ such that
\begin{equation}
v(t) < \eta(t),\qquad  t_0\le t \le t_0 + \delta.
\end{equation}
To continue the proof we need two Claims.

{\it Claim 1.}  {\it If} 
\begin{equation}
v(t) \le \eta(t),\qquad \forall t \in[ t_0, T],\quad T>t_0,
\end{equation}
{\it then} 
\begin{equation}
\dot{v}(t) \le \dot{\eta}(t),\qquad  \forall t \in[ t_0, T].
\end{equation}
{\it Proof of Claim 1.}
 
It follows from inequalities \eqref{1eq4}, \eqref{eq6} 
and the inequality $v(T)\leq \eta(T)$,  that
\begin{equation}
\label{eq14}
\begin{split}
\dot{v}(t) &\le 
e^{\int_{t_0}^t\gamma(s)ds}\alpha(t,\frac{1}{\mu(t)})
+ \beta(t)e^{\int_{t_0}^t\gamma(s)ds}\\
&\le \frac{e^{\int_{t_0}^t\gamma(s)ds}}{\mu(t)}
\bigg{[}\gamma -\frac{\dot{\mu}(t)}{\mu(t)}\bigg{]}\\
& = \frac{d}{dt} \frac{e^{\int_{t_0}^t\gamma(s)ds}}
{\mu(t)}\bigg{|}_{t=t} = \dot{\eta}(t),\qquad \forall t\in [t_0,T].
\end{split}
\end{equation}
{\it Claim 1} is proved. $\hfill$ $\Box$

Denote 
\begin{equation}
\label{eq12}
T:=\sup \{\delta \in\mathbb{R}^+: v(t) < \eta(t),\, 
\forall t \in[t_0, t_0 + \delta]\}.
\end{equation}
{\it Claim 2.}  {\it One has $T=\infty$.} 

Claim 2 says that every nonnegative solution $g(t)$ to inequality (1),
satisfying assumption (3),  is defined 
globally.

{\it Proof of Claim 2.}

Assume the contrary, i.e.,  $T<\infty$. The solution $v(t)$
to \eqref{eq6}
is continuous at every point $t$ at which it is bounded.
From the definition of $T$ and the continuity of $v$ and $\eta$ one 
gets
\begin{equation}
\label{eq13}
v(T) \le \eta(T).
\end{equation}
It follows from 
inequalities \eqref{eq12}, \eqref{eq13},  and {\it Claim 1} that 
\begin{equation}
\label{2eq18}
\dot{v}(t)\le \dot{\eta}(t),\qquad \forall t\in [t_0,T].
\end{equation}
This implies 
\begin{equation}
\label{2eq19}
v(T)-v(t_0) = \int_{t_0}^T \dot{v}(s)ds \le \int_{t_0}^T \dot{\eta}(s)ds 
= \eta(T)-\eta(t_0).
\end{equation}
Since $v(t_0)<\eta(t_0)$ by assumption \eqref{1eq6}, it follows from inequality \eqref{2eq19} that
\begin{equation}
\label{eq20}
v(T) < \eta(T).
\end{equation}
Inequality \eqref{eq20} and inequality \eqref{2eq18} with $t=T$ imply that
 there exists a $\delta>0$ such that
\begin{equation}
v(t) < \eta(t),\qquad  T\le t \le T + \delta.
\end{equation}
This contradicts the definition of $T$ in \eqref{eq12}, and the 
contradiction proves the desired conclusion $T=\infty$. 

Claim 2 is proved. $\hfill$ $\Box$ 

It follows from the definitions of $\eta(t)$, $T$,  $v(t)$, and from 
the relation $T=\infty$, that
\begin{equation}
g(t) = e^{-\int_{t_0}^t\gamma(s)ds} v(t)< 
e^{-\int_{t_0}^t\gamma(s)ds}\eta(t) = \frac{1}{\mu(t)},\qquad \forall t> t_0.
\end{equation}
Theorem~\ref{thm1} is proved.
\end{proof}

\begin{thm} \label{theoremeq} 
Let $\beta(t)$ and $\gamma(t)$ be as in
Theorem~\ref{thm1}.  Assume that $0<\alpha(t,y)$ is continuous with
respect to $t$ on $[t_0,\infty)$, is locally Lipschitz-continuous and
nondecreasing with respect to $y$ on $[0,\infty)$.  Let $0\le g(t)$
satisfy \eqref{neq1} and $0<\mu(t)$ satisfy \eqref{1eq4} and 
\eqref{1eq6}.
Then \begin{equation} \label{eq.2} g(t)\le
\frac{1}{\mu(t)},\qquad \forall t\ge t_0. \end{equation} \end{thm}

\begin{proof}
Let $v(t)$ be defined in \eqref{eqqe1}. Then inequality \eqref{eq6} holds. 
%
Let $w_n(t)$ solve the following differential equation 
\begin{equation}
\label{eq.4}
\dot{w}_n(t) = a(t)\alpha\big{(}t,w_n(t)e^{-\int_{t_0}^t\gamma(s)ds}\big{)} 
+ b(t), \qquad w_n(t_0) = g(t_0) - \frac{1}{n},\qquad n\ge n_0,
\end{equation}
where $n_0$ is sufficiently large and $g(t_0) >\frac{1}{n_0}$. 
Since $\alpha(t,y)$ is continuous with respect to $t$ and 
locally Lipschitz-continuous with respect to $y$, 
there exists a unique local solution to \eqref{eq.4}.

From the proof of Theorem~\ref{thm1} one gets
\begin{equation}
\label{eq.5}
w_n(t) < \frac{e^{\int_{t_0}^t\gamma(s)ds}}{\mu(t)},
\qquad \forall t\ge t_0, \forall n\ge n_0.
\end{equation}
Let $t_0<\tau<\infty$ be an arbitrary constant and 
\begin{equation}
\label{eq.6}
w(t) = \lim_{n\to\infty} w_n(t),\qquad \forall t \in [t_0,\tau].
\end{equation}
This and the fact that $w_n(t)$ is uniformly continuous 
on $[0,\tau]$ imply that $w(t)$ solves the following equation:
\begin{equation}
\label{eq.7}
\dot{w}(t) = a(t)\alpha\big{(}t,w(t)e^{-\int_{t_0}^t\gamma(s)ds}\big{)} 
+ b(t), \qquad w(t_0) = g(t_0),\qquad \forall t\in [0,\tau]. 
\end{equation}
Note that the solution $w(t)$ to \eqref{eq.7} is unique since 
$\alpha(t,y)$ is continuous with respect to $t$ and 
locally Lipschitz-continuous with respect to $y$. 
From \eqref{eq6}, \eqref{eq.7}, 
a comparison lemma (see, e.g.,\,\cite{R499},\, p.99), 
 the continuity of $w_n(t)$ with respect to $w_0(t_0)$ on
 $[0,\tau]$, and \eqref{eq.5}, one gets
\begin{equation}
\label{eq.8}
v(t) \le w(t) \le \frac{e^{\int_{t_0}^t\gamma(s)ds}}{\mu(t)},
\qquad \forall t\in [t_0, \tau],\,\,  \forall n\ge n_0.
\end{equation} 
Since $\tau> t_0$ is arbitrary, inequality \eqref{eq.2} follows from 
\eqref{eq.8}. 

Theorem~\ref{theoremeq} is proved. 
\end{proof}

Let us consider a {\it discrete analog} of Theorem \ref{thm1}.

Let
$$
\frac{g_{n+1}-g_n}{h_n}\le -\gamma_n g_n+\alpha(n,g_n) +
\beta_n,\qquad h_n > 0,\quad 0< h_n\gamma_n < 1,\quad p>1,
$$
and the inequality:
$$
g_{n+1} \le (1-\gamma_n) g_n + \alpha(n,g_n) + \beta_n,\quad n\ge 0,
\qquad 0<\gamma_n<1,\quad p>1,
$$
where $g_n, \beta_n$ and $\gamma_n$ are positive sequences
of real numbers.

Under suitable assumptions on $\beta_n$ and $\gamma_n$, 
we obtain an upper bound for $g_n$ as $n\to\infty$. In particular, we give 
sufficient conditions for the validity of the relation 
$\lim_{n\to\infty}g_n=0$, and estimate the rate 
of growth/decay of $g_n$ as $n\to\infty$. This result can be used in a study of 
evolution problems, 
in a study of iterative processes, and in a study of nonlinear PDE.

\begin{thm}
\label{thm2}
Let $\beta_n,$ and $g_n$ be nonnegative sequences 
of numbers. Assume that
\begin{equation}
\label{eq1}
\begin{split}
\frac{g_{n+1}-g_n}{h_n}&\le -\gamma_n g_n+
\alpha(n,g_n) +\beta_n,\qquad h_n > 0,\quad 0< h_n\gamma_n < 1,
\end{split}
\end{equation}
or, equivalently, 
\begin{equation}
\qquad g_{n+1}\le g_n(1-h_n\gamma_n) +
 h_n \alpha(n,g_n)+h_n\beta_n,\qquad h_n > 0,\quad 0< h_n\gamma_n < 1.
\end{equation}
If there is a sequence of positive numbers 
$(\mu_n)_{n=1}^\infty$, such that the following conditions hold:
\begin{align}
\label{eq3}
\alpha(n,\frac{1}{\mu_n})+\beta_n&\le \frac{1}{\mu_n}\bigg{(}\gamma_n -
\frac{\mu_{n+1}-\mu_n}{\mu_n h_n}\bigg{)},\\
\label{eq2}
g_0&\le\frac{1}{\mu_0},
\end{align}
then
\begin{equation}
\label{eq5}
0\leq g_n\le\frac{1}{\mu_n} \qquad \forall n\ge 0.
\end{equation}
Therefore, if $\lim_{n\to\infty}\mu_n =\infty$, then $\lim_{n\to\infty} 
g_n = 0$.
\end{thm}

\begin{proof}
Let us prove \eqref{eq5} by induction. Inequality \eqref{eq5} holds for $n=0$ by assumption \eqref{eq2}. 
Suppose that \eqref{eq5} holds for all $n\le m$. From inequalities 
\eqref{eq1}, \eqref{eq3}, and from the induction 
hypothesis 
$g_n\le\frac{1}{\mu_n}$, $n\le m$, one gets
\begin{equation}
\begin{split}
g_{m+1}&\le g_m(1-h_m\gamma_m) + h_m \alpha(m,g_m) + h_m\beta_m\\
\le& \frac{1}{\mu_m}(1-h_m\gamma_m)+ h_m\alpha(m,\frac{1}{\mu_m})+ h_m\beta_m\\
\le& \frac{1}{\mu_m}(1-h_m\gamma_m) + \frac{h_m}{\mu_m}\bigg{(}\gamma_m -
\frac{\mu_{m+1}-\mu_m}{\mu_m h_m}\bigg{)}\\
=& \frac{1}{\mu_{m+1}}- \frac{\mu_{m+1}^2 -2\mu_{m+1}\mu_{m}+\mu_m^2}{\mu_n^2 \mu_{m+1}} 
\le\frac{1}{\mu_{m+1}}.
\end{split}
\end{equation}
Therefore, inequality \eqref{eq5} holds for $n=m+1$. 
Thus, inequality \eqref{eq5} holds for all $n\ge 0$ by induction. 
Theorem~\ref{thm2} is proved.
\end{proof}

Setting $h_n =1$ in Theorem~\ref{thm2}, one obtains the following result:
\begin{thm}
\label{cor1}
Let $\beta_n,\gamma_n$ and $g_n$ be sequences of nonnegative 
numbers, and
\begin{equation}
\label{2eq1}
\begin{split}
g_{n+1}&\le g_n(1-\gamma_n) +\alpha(n,g_n)+\beta_n,\qquad 0<\gamma_n <1.
\end{split}
\end{equation}
If there is sequence $(\mu_n)_{n=1}^\infty>0$ such 
that the following conditions hold
\begin{align}
g_0\le\frac{1}{\mu_0},\qquad\alpha(n,\frac{1}{\mu_n})+\beta_n&\le \frac{1}{\mu_n}\bigg{(}\gamma_n -
\frac{\mu_{n+1}-\mu_n}{\mu_n h_n}\bigg{)},\qquad \forall n\ge 0,
\end{align}
then
\begin{equation}
\label{2eq5}
g_n\le\frac{1}{\mu_n}, \qquad \forall n\ge 0.
\end{equation}
\end{thm}

\section{Applications}
\label{sec3}

Here we  sketch an idea for possible applications of our 
inequalities
in a study of dynamical systems in a Hilbert space $H$.

In this Section we assume without loss of generality that $t_0=0$. Let   
\begin{equation}
\label{eqr1}
\dot{u} +Au=  h(t,u) + f(t),\qquad u(0)= u_0,
\quad \dot{u}:=\frac{d u}{d t},\quad t\ge 0.
\end{equation}
To explain the ideas, let us make simplifying assumptions:  $A >0$ is a 
selfadjoint 
time-independent operator in a real Hilbert 
space $H$, 
$h(t,u)$ is a nonlinear operator in $H$, locally Lipschitz with 
respect to $u$
and continuous with respect to $t\in \mathbb{R}_+:=[0,\infty)$,
and $f$ is a continuous function on $\mathbb{R}_+$ with values in $H$,
$\sup_{t\geq 0}\|f(t)\|<\infty$.
The scalar product in $H$ is denoted $\langle u,v \rangle$.
Assume that
\begin{equation}
\label{eqr2}
\langle Au,u \rangle \ge \gamma \langle u,u \rangle, \quad 
\gamma=const>0,
\quad \|h(t,u)\| \le \alpha(t,\|u\|), \qquad 
\forall u\in D(A),
\end{equation}
where 
$\alpha(t,y)\leq c|y|^p$, $p>1$ and $c>0$ are  constants, 
and $\alpha(t,y)$ is a  non-decreasing $C^1([0,\infty))$ function of 
$y$. Our approach works when $\gamma=\gamma(t)$ and $c=c(t)$, see 
Examples 1,2 below.
The problem is to estimate the behavior of the solution to \eqref{eqr1} 
as
$t\to\infty$ and to give sufficient conditions for a global
existence of the unique solution to \eqref{eqr1}. Our approach consists of
a reduction of this problem to the inequality \eqref{neq1} and an
application of Theorem \ref{thm1}. A different approach, studied in
the literature (see, e.g., \cite{K},  \cite{P}), is based on the semigroup 
theory.

Let $g(t):=\|u(t)\|$. Problem \eqref{eqr1} has a unique local solution 
under our assumptions. This local solution exists globally if
$\sup_{t\geq 0}\|u(t)\|<\infty$.
Multiply \eqref{eqr1} by $u$ 
and use \eqref{eqr2} to get
\begin{equation}
\label{eqr3}
\dot{g}g\le -\gamma(t)g^2 + \alpha(t,g) g + \beta(t)g,
\qquad \beta(t):=\|f(t)\|.
\end{equation}
Since $g\ge 0$, one gets
\begin{equation}
\label{eqr4}
\dot{g} \le -\gamma(t)g +  \alpha(t,g(t)) + \beta(t).
\end{equation}

Now Theorem~\ref{thm1} is applicable and yields 
sufficient conditions \eqref{1eq4} and \eqref{1eq6} for the global
existence of the solution to \eqref{eqr1} and estimate \eqref{3eq10} for the 
behavior of $\|u(t)\|$ as $t\to\infty$. The choice of $\mu(t)$ in
Theorem~\ref{thm1} is often straightforward. For example, if
$\alpha(t,g(t))=\frac{c_0}{a(t)} g^2$, where $\lim_{t \to \infty} a(t)=0$, 
$\dot{a}(t)<0$, then one can often choose $\mu(t)=\frac {\lambda}{a(t)}$,
$\lambda=const>0$, see \cite{R499}, p.116, and \cite{R549}, p.487, for  
examples of applications of this approach. 

The outlined approach is applicable to stability of the 
solutions to nonlinear differential equations, 
to semilinear parabolic problems, to hyperbolic problems, and
other problems. There is a large literature on the stability
of the solutions to differential equations (see, e.g., \cite{DK}, 
\cite{D}, and references therein).
Our approach yields some novel results. If the selfadjoint operator 
$A$ depends on $t$, $A=A(t)$, and $\gamma=\gamma(t)>0$,
$\lim_{t\to \infty}\gamma(t)=0$, 
one can treat problems with degenerate, as $t\to \infty$, 
elliptic operators $A$. 

For instance, if the operator $A$
is a second-order elliptic operator with matrix $a_{ij}(x,t)$,
and the minimal eigenvalue  $\lambda(x,t)$ of this matrix 
satisfies the condition $\min_{x}\lambda(x,t):=\gamma (t)\to 0$ 
as $t\to \infty$, then Theorem~\ref{thm1} is applicable
under suitable assumptions on $\gamma(t)$, $h(t,u)$ and $f(t)$.

{\it Example 1.} Consider 
\begin{equation}
\label{eqr5}
\dot{u}= - \gamma(t) u +  a(t)u(t)|u(t)|^p + \frac 1{(1+t)^q},\qquad 
u(0)=0,
\end{equation}
where $\gamma(t)=\frac c{(1+t)^b}$, $a(t)=\frac 1{(1+t)^m}$,  $p$, $q$, 
$b$, $c$, and $m$ are positive constants.
Our goal is to give sufficient conditions for the solution to the above 
problem to converge to zero as $t\to \infty$.
Multiply   \eqref{eqr5} by $u$, denote $g:=u^2$, and get the following 
inequality
\begin{equation}
\label{eqr6}
\dot{g}\leq -2 \frac{c}{(1+t)^b} g +  2\frac{1}{(1+t)^m}g(t)^{1+0.5p} + 
2\frac{1}{(1+t)^q}g^{0.5}, \qquad g=u^2.
\end{equation}
Choose $\mu(t)=\lambda(1+t)^{\nu}$, where $\lambda>0$ and $\nu>0$
are constants. 

Inequality  \eqref{1eq4} takes the form:
\begin{equation}
\label{eqr7}
\begin{split}
\frac 
2{(1+t)^m}[\lambda(1+t)^{\nu}]^{-1-0.5p}&+\frac2{(1+t)^q}[\lambda(1+t)^{\nu}]^
{-0.5}\\
&\le 
[\lambda(1+t)^{\nu}]^{-1}\bigg{(}2\frac c{(1+t)^b}-\frac \nu{1+t}\bigg{)}. 
\end{split}
\end{equation}
Choose $p,q,m,c, \lambda$ and $\nu$ so that 
inequality \eqref{eqr7} be valid and $\lambda u(0)^2<1$, so that
condition \eqref{1eq6} with $t_0=0$ holds. If this 
is done,
then $u^2(t)\leq \frac 1 {\lambda(1+t)^{\nu}}$, so $\lim_{t\to 
\infty}u(t)=0$.
For example, choose  $b=1$, $\nu=1$, $q=1.5$, $m=1$, $\lambda=4$,
 $c=4$, $p\geq 1$. Then inequality \eqref{eqr7} is valid, and if 
$u(0)^2<1/4$, then \eqref{1eq6} with $t_0=0$ holds, so
$\lim_{t\to \infty}u(t)=0$. The choice of the parameters can be varied.
In particular, the nonlinearity growth, governed by $p$, can be arbitrary
in power scale. If $b=1$ then three inequalities $m+0.5 p\nu\ge 1$, 
$q-0.5\nu\ge 1$, and $\lambda^{1/2}+\lambda^{-0.5p}\le c-0.5\nu$ 
together with $u(0)^2<\lambda^{-1}$ are sufficient for \eqref{1eq6}
and \eqref{eqr7} to hold, so they imply $\lim_{t\to \infty}u(t)=0$.

{\it Example 2.} 
Consider problem \eqref{eqr1} with $A, h$ and $f$ satisfying 
\eqref{eqr2} with $\gamma\equiv 0$. 
So one gets inequality \eqref{eqr4} with $\gamma(t)\equiv 0$. Choose
\begin{equation}
\label{eqr10}
\mu(t):= c + \lambda(1+ t)^{-b},\qquad c>0,\, b>0,\, \lambda>0,
\end{equation}
where $c, \lambda, $ and $b$ are constants.
Inequality  \eqref{1eq4} takes the form:
\begin{equation}
\label{eqr11}
\alpha(t, \frac{1}{\mu(t)}) + \beta(t) \le \frac{1}
{\mu(t)}\frac{b\lambda }{(1+t)[\lambda + c(1+t)^b]}.
\end{equation}
Let $\theta\in (0,1)$, $p>0$, and $C>0$ be constants. Assume that 
\begin{equation}
\label{eqr12}
\alpha(t,|y|) \le \theta C |y|^p\frac{b\lambda }{(\lambda + c)
(1+t)^{1+b}},\qquad \beta(t) \le (1-\theta)\frac{b\lambda }{(c+\lambda)^2(1+t)^{1+b}},
\end{equation}
for all $t\ge 0$, and 
\begin{equation}
\label{eqex1}
C  = 
\left \{\begin{matrix} c^{p-1} &\quad \text{if}\qquad p>1,\\
 (\lambda+c)^{p-1}&\quad \text{if}\qquad p\le 1.
 \end{matrix} \right .
\end{equation}
Let us verify that inequality \eqref{eqr11} holds given that \eqref{eqr12} and \eqref{eqex1} hold. 

It follows from \eqref{eqr10} that $c<\mu(t)\le c+\lambda$, $\forall t\ge 0$. 
This and \eqref{eqr12} imply
\begin{equation}
\label{eqtt1}
\beta(t)\le (1-\theta) 
\frac{1}{(c+\lambda)^2(1+t)^{1+b}}\le (1-\theta)\frac{1}{\mu(t)} \frac{1}{(1+t)(c+\lambda(1+t)^b)}.
\end{equation}
From \eqref{eqex1} and \eqref{eqr10} one gets
\begin{equation}
\label{eqt3}
\frac{C}{\mu^{p-1}(t)} \le C \max (c^{1-p}, (c+\lambda)^{1-p}) \le 1,\qquad \forall t\ge 0.
\end{equation}
From \eqref{eqr12} and \eqref{eqt3} one obtains
\begin{equation}
\label{eqt2}
\begin{split}
\alpha(t,\frac{1}{\mu(t)})
&\le
\theta C \frac{1}{\mu(t)} \frac{1}{\mu^{p-1}(t)} \frac{b\lambda }{(1+t)[ \lambda(1+t)^{b} + c
(1+t)^{b}]}\\
&\le \theta 
\frac{1}{\mu(t)} \frac{b\lambda }{(1+t)[ \lambda + c
(1+t)^{b}]}.
\end{split}
\end{equation}
Inequality \eqref{eqr11} follows from \eqref{eqtt1} and \eqref{eqt2}. 
From \eqref{eqr11} and Theorem~\ref{thm1} one obtains
\begin{equation}
\label{eqr14}
g(t)\le \frac{1}{\mu(t)} < \frac{1}{c},\qquad \forall t>0, 
\end{equation}
provided that  $g(0)< (c+\lambda)^{-1}$. 
From \eqref{eqr4} with $\gamma(t)=0$ and \eqref{eqr12}--\eqref{eqr14}, one gets
$\dot{g}(t) = O(\frac{1}{(1+t)^{1+b}})$. Thus, there exists  
finite limit $\lim_{t\to \infty}g(t) = g(\infty)\le c^{-1}$.


\begin{thebibliography}{99}
\small


\bibitem{DK} Yu. L. Daleckii,  M. G. Krein, Stability of solutions of 
differential equations in Banach spaces, Amer. Math. Soc., Providence, RI, 
1974.

\bibitem{D} B. P. Demidovich, Lectures on stability theory, 
Nauka, Moscow, 1967 (in Russian)

\bibitem{R549} N. S. Hoang, A. G. Ramm, Dynamical Systems Gradient 
Method for solving 
nonlinear operator equations with monotone 
operators, Acta Appl. Math., 106, (2009) , 473-499. 

\bibitem{R538}
N.S. Hoang and A. G. Ramm, A nonlinear inequality, 
Jour. Math. Ineq., 2, N4, (2008), 459-464. 

\bibitem{R558}
N.S. Hoang and A. G. Ramm, 
A nonlinear inequality and applications, Nonlinear Analysis: 
Theory, Methods \& Applications, 71, (2009), 2744- 2752.

\bibitem{K} S. G. Krein, Linear differential equations in Banach spaces, 
Amer. Math. Soc., Providence, RI, 1971.

\bibitem{P} A. Pazy, Semigroups of linear operators and 
applications to partial differential equations,
Springer-Verlag, New York, 1983.


\bibitem{R499} A. G. Ramm, Dynamical systems method for solving
operator equations, Elsevier, Amsterdam, 2007.

\end{thebibliography}
\end{document}